\documentclass[11pt]{article}
\usepackage{amssymb, amsmath, amsfonts, amsthm, graphics, mathrsfs, enumerate}
\usepackage{hyperref}
\hypersetup{
    colorlinks=true,
    linkcolor=blue,
    filecolor=magenta,      
    urlcolor=cyan,
}
\usepackage[hmargin=1 in, vmargin = 1 in]{geometry}

\usepackage{graphicx}

\renewcommand{\d}{\mathrm{d}}
\newcommand{\xx}{\mathbf{x}}
\renewcommand{\d}{\mathrm{d}}
\newcommand{\RR}{\mathbb{R}}
\newcommand{\CC}{\mathbb{C}}

\newtheorem{theorem}{Theorem}
\newtheorem{corollary}[theorem]{Corollary}
\newtheorem{lemma}[theorem]{Lemma}
\newtheorem{remark}{Remark}

\graphicspath{{fig/}}

\title{An Extension of the Euler-Maclaurin Summation Formula to Nearly Singular Functions}
\author{Bowei Wu$^1$}
\date{$^1$University of Massachusetts Lowell\\[2ex]
	\today}

\begin{document}

\maketitle

\begin{abstract}
A extension of the Euler-Maclaurin (E-M) formula to near-singular functions is presented. 
This extension is derived based on earlier generalized E-M formulas for singular functions.
The new E-M formulas consists of two components: a ``singular'' component that is a continuous extension of the earlier singular E-M formulas, and a ``jump'' component associated with the discontinuity of the integral with respect to a parameter that controls near singularity. The singular component of the new E-M formulas is an asymptotic series whose coefficients depend on the Hurwitz zeta function or the digamma function. Numerical examples of near-singular quadrature based on the extended E-M formula are presented, where accuracies of machine precision are achieved insensitive to the strength of the near singularity and with a very small number of quadrature nodes.
\end{abstract}


\section{Introduction}
\label{sc:intro}

This paper considers simple and efficient numerical approximations of integrals of the form
\begin{equation}
\int_{-a}^a\frac{g(x)}{d^2+c^2(x-x_s)^2}\,\d x
\label{eq:nearsing}
\end{equation}
where $a,c,d>0$ and $x_s\in(-a,a)$ are constants, and where $g\in C^\infty([-a,a])$. The regularity constraint on $g$ can be relaxed. When $d\approx0$, the integrand is nearly singular at $x=x_s$, making the integral difficult to approximate using non-adaptive regular quadrature rules.

The consideration of \eqref{eq:nearsing} is motivated by the numerical challenge of evaluating integral operators of the form 
\begin{equation}
\int_S K(\tilde{\xx},\xx)g(\xx)\d s_\xx
\label{eq:nearsing_op}
\end{equation}
at a target point $\tilde{\xx}$ that is at a distance $d\approx0$ from the surface $S$, which leads to near-singularity of the type \eqref{eq:nearsing} in the kernel $K(\tilde{\xx},\xx)$ \cite{kress2014linear}. The accurate and efficient evaluation of near-singular integrals is important in numerical simulations of physical systems that involve closely interacting components, such as viscous particulate flows or potential flows \cite{lsc2d,nitsche2021tcfd}.

The Trapezoidal rule with error corrections has proven to be effective for the discretization of integral operators such as \eqref{eq:nearsing_op} \cite{sidi1988quadrature,kapur1997high,alpert1999hybrid,wu2021zeta}. The theoretical bases for these error corrections are the Euler-Maclaurin formula and its generalized forms. The classical Euler-Maclaurin formula \cite{euler1741inventio,maclaurin1742treatise} 
provides an asymptotic error expansion for the trapezoidal quadrature rule applied to regular integrals. Starting from the 1960s, the Euler-Maclaurin formula has been generalized for integrals with integrable algebraic and logarithmic endpoint singularities by Navot \cite{navot1961extension,navot1962further}, for Hadamard finite-part integrals by Ninham and Lyness \cite{ninham1966generalised, lyness1994finite}, and finally for the $1/x$ type finite-part integrals by Sidi  \cite{sidi2012algebraic}, who has also further extended the Euler-Maclaurin formula for polynomial compositions of algebraic and logarithmic endpoint singularities \cite{sidi2004euler}.

In this paper, we provide a new extension of the Euler-Maclaurin formula for nearly singular integrals of the form \eqref{eq:nearsing}, which is derived from a combination of the above-mentioned singular Euler-Maclaurin formulas and the rational zeta series that draw connections between special functions. One unexpected consequence is an ``off-mesh'' Euler-Maclaurin formula for a hypersingular (i.e. finite-part) integral that is not seen in the literature before (see Corollary \ref{thm:sing_EM_offmesh}).  It is also worth mentioning that the current work is closely connected to the work of Nitsche \cite{nitsche2021tcfd}, in which the near-singular error corrections for the Trapezoidal rule are computed numerically (instead of using the analytic Euler-Maclaurin type expressions). While the results of the current paper provide more theoretical insights into the error analysis and makes the error correction easier to compute, the approach of Nitsche \cite{nitsche2021tcfd} can be more easily generalized to integrals in higher dimensions and has recently been applied to simulations of viscous fluid flows past ellipsoidal objects \cite{nitsche2025corrected}.

This paper is organized as follows. In Section \ref{sc:EMsidi}, we first set up the notations and review the existing generalized Euler-Maclaurin formulas for singular integrals, which serves as the basis for approximating near-singular integrals. Then Section \ref{sc:err_flat} presents the main result of this paper: a generalized Euler-Maclaurin formula for near-singular integrals of the form \eqref{eq:nearsing} and its analysis. Finally, numerical examples are given in Section \ref{sc:numerical}, convergence of the new Euler-Maclaurin formula is demonstrated and compared with the classical regular Euler-Maclaurin formula. Further discussions and conclusions are in Section \ref{sc:conclude}.


\section{Preliminary: Euler-Maclaurin formulas for singular functions}
\label{sc:EMsidi}

We first set up the notations that will be used throughout the paper. 
Let $f$ be a function that is $C^\infty$ on $\RR$ except for a possible singularity at $x=0$, where the $C^\infty$ regularity can be relaxed to $C^p$ for a sufficiently large $p$. Let $a>0$ and consider the following integrals
\begin{align}
I[f] &:= \int_{-a}^af(x)\,\d x, & I_L[f] &:= \int_{-a}^0f(x)\,\d x, & I_R[f] &:= \int_{0}^af( x)\,\d x,
\label{eq:notation_int}
\end{align}
and their corresponding \emph{punctured} Trapezoidal rules (where the possible singular point $x=0$ is omitted from the summation)
\begin{align}
L_h[f] &:= \sum_{k=-n}^{-1}f(kh)\, h + C_h^{-a}, & R_h[f] &:= \sum_{k=1}^{n}f(kh)\, h + C_h^a, & T_h[f] &:= L_h[f] + R_h[f],
\label{eq:notation_trap}
\end{align}
where $h =\frac{a}{n}$, and where $C_h^{-a}$ and $C_h^a$ are edge corrections at $x=-a$ and $x=a$, respectively, from the classical Euler-Maclaurin formula. For example, $C_h^a$ has the form
\begin{equation}
C_h^a = -\frac{1}{2}f(a)h - \sum_{k=1}^{\lfloor p/2\rfloor}\frac{B_{2k}}{(2k)!}f^{(2k-1)}(a)h^{2k}
\end{equation}
where $B_{2k}$ are the Bernoulli numbers and where $p$ is sufficiently large such that the dominant quadrature error comes from the singularity of $f$ at $x=0$; alternatively, $C_h^a$ can also be approximated using Gregory interpolants \cite{javed2016euler}. An implementation of the Gregory interpolants is available in the Matlab code for the numerical examples in this paper, see Remark \ref{rmk:code}.

The error of the punctured Trapezoidal rule is
\begin{align}
E_h[f] := I[f] - T_h[f].
\label{eq:notation_err}
\end{align}
In addition, we also define the shifted Trapezoidal rules (such that $x=0$ is not a mesh point)
\begin{equation}
\begin{aligned}
L_{h,s}[f] &:= \sum_{k=-n}^{-1}f(k\,h+s\,h)\, h + C_{h}^{-a+s}\\
R_{h,s}[f] &:= \sum_{k=1}^{n}f(k\,h+s\,h)\, h + C_{h}^{a+s},\\
T_{h,s}[f] &:= L_{h,s}[f] + R_{h,s}[f].
\end{aligned}
\label{eq:notation_trap_shift}
\end{equation}
where the shift is $s\,h$ for $|s|\leq\frac{1}{2}$; note that when $s=0$, \eqref{eq:notation_trap_shift} reduces to \eqref{eq:notation_trap}.

The following analytical tool will be critical for our analysis in the next section.

\begin{theorem}[generalized Euler-Maclaurin formula]
\label{thm:gen_EM}
Let $w\in\CC$, then
\begin{align}
&R_{h,s}\left[\frac{g( x)}{ x^w}\right] - f.p.\,I_R\left[\frac{g( x)}{ x^w}\right] \sim \sum_{k=0}^\infty \zeta_h(w-k,1+s)\frac{g^{(k)}(0)}{k!}h^{k+1-w}\quad \text{as }h\to0,\\
&L_{h,s}\left[\frac{g( x)}{ x^w}\right] - f.p.\,I_L\left[\frac{g( x)}{ x^w}\right] \sim \sum_{k=0}^\infty \zeta_h(w-k,1-s)(-1)^k\frac{g^{(k)}(0)}{k!}h^{k+1-w}\quad \text{as }h\to0,
\end{align}
where ``$f.p.$'' indicates that an integral is understood as a Hadamard finite-part integral, and where $\zeta_h(z,s)$ is a \emph{modified Hurwitz zeta function} given by
\begin{equation}
\zeta_h(z,s) = \begin{cases}
\zeta(z,s), &z\neq1\\
-\psi(s)-\log h, &z=1
\end{cases}.
\label{eq:modified_hurwitzzeta}
\end{equation}
where $\zeta(z,s)=\sum_{n=0}^\infty(n+s)^{-z}$ is the classical \emph{Hurwitz zeta function} (which has an analytic extension for all $z\in\CC\setminus\{1\}$), and where $\psi(s)={\Gamma'(s)}/{\Gamma(s)}$ is the digamma function and $\Gamma(s)$ is the Euler's gamma function.%
\end{theorem}
For a proof of Theorem 1, see \cite{navot1961extension}. See also \cite{davis1984methods} for a summary of the properties of Hadamard finite-part integrals.

As a special case of Theorem \ref{thm:gen_EM}, when $s=1$ we have the following corollary, whose proof can be found in \cite{sidi2012algebraic}.

\begin{corollary}
\label{thm:gen_EM_special}
\begin{align}
&R_h\Big[\frac{g( x)}{ x^w}\Big] - f.p.\,I_R\Big[\frac{g( x)}{ x^w}\Big] \sim \sum_{k=0}^\infty \zeta_h(w-k)\frac{g^{(k)}(0)}{k!}h^{k+1-w}\quad \text{as }h\to0,\\
&L_h\Big[\frac{g( x)}{ x^w}\Big] - f.p.\,I_L\Big[\frac{g( x)}{ x^w}\Big] \sim \sum_{k=0}^\infty \zeta_h(w-k)(-1)^k\frac{g^{(k)}(0)}{k!}h^{k+1-w}\quad \text{as }h\to0,\\
&T_h\Big[\frac{g( x)}{ x^w}\Big] - f.p.\,I\Big[\frac{g( x)}{ x^w}\Big] \sim \sum_{k=0}^\infty 2\zeta_h(w-2k)\frac{g^{(2k)}(0)}{(2k)!}h^{2k+1-w}\quad \text{as }h\to0,
\end{align}
where $\zeta_h(z)$ is a \emph{modified Riemann zeta function} given by
\begin{equation}
\zeta_h(z) = \begin{cases}
\zeta(z), &z\neq1\\
\gamma-\log h, &z=1
\end{cases}.
\label{eq:modified_zetafunc}
\end{equation}
where $\zeta(z)$ is the classical \emph{Riemann zeta function}, and where $\gamma = -\psi(1)$ is the Euler's constant.
\end{corollary}

\section{An Euler-Maclaurin formula for near-singular functions}
\label{sc:err_flat}

The central focus of this section is the Trapezoidal quadrature approximation of the integral
\begin{equation}
I\Big[\frac{g( x)}{d^2+c^2( x- x_s)^2}\Big] = \int_{-a}^a\frac{g( x)}{d^2+c^2( x- x_s)^2}\,\d x
\label{eq:nearsingint}
\end{equation}
where $c,d>0$ are constants, $g( x)$ a smooth function compactly supported on $[-a,a]$, and $ x_s\in(-a,a)$ the location of near singularity.  We will analyze the error of the punctured Trapezoidal rule $T_h\Big[\frac{g( x)}{d^2+c^2( x- x_s)^2}\Big]$; see Figure \ref{fig:punctured_trap_flat} for a schematic of the punctured Trapezoidal rule.
Our analysis begins with the case $ x_s=0$ such that the near-singularity is at a mesh point, namely $x=0$. Then we will consider the general case of $ x_s\neq0$. In both cases, we will show that the error $E_h\Big[\frac{g( x)}{d^2+c^2( x- x_s)^2}\Big]$ consists of two components: a ``singular error'' which is a continuous extension of the generalized Euler-Maclaurin formulas from Section \ref{sc:EMsidi}, and a ``jump error'' associated with the discontinuity of the value of the integral with respect to $d$.

\begin{figure}[htbp]
\centering
\includegraphics[scale=0.45]{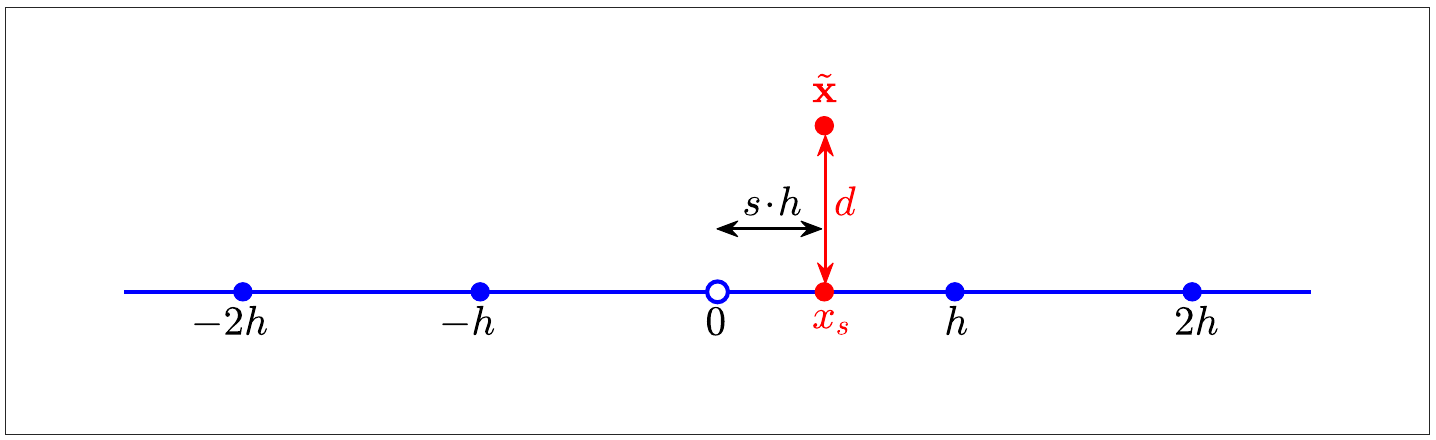}
\caption{Schematic of the punctured trapezoidal rule of mesh size $h$ for the near-singular integral \eqref{eq:nearsingint}, which can arise from an integral operator \eqref{eq:nearsing_op} with its target point $\tilde\xx$ located at a distance $d$ from the surface $S$, whose orthogonal projection onto $S$ corresponds to $x_s=s\cdot h$ in the parameter space relative to the closest mesh point $x=0$. The open circle at $x=0$ indicates its omission in the summation of the punctured Trapezoidal rule.}
\label{fig:punctured_trap_flat}
\end{figure}

\subsection{Near-singular quadrature error analysis}
\label{sc:nearsing_onmesh}

\begin{theorem}[Euler-Maclaurin for near-singular integral] 
\label{thm:nearsing_EM}
Given $g(x)\in C^{\infty}([a,a])$ and $c,d>0$, and using the notations (\ref{eq:notation_int}-\ref{eq:notation_err}), we have
\begin{align}
E_h\Big[\frac{g( x)}{d^2+c^2 x^2}\Big] &\sim -\frac{1}{c^2}\sum_{k=0}^\infty\frac{g^{(2k)}(0)}{(2k)!}h^{2k-1}(2\,z_{2k}) + \frac{\pi}{cd}\sum_{k=0}^\infty \frac{g^{(2k)}(0)}{(2k)!}\left(-\frac{d^2}{c^2}\right)^{k}
\label{eq:nearsing_EM}
\end{align}
as $h\to0$ and $d\to0$, where the numbers $z_k$ satisfy the recurrence relation
\begin{equation}
z_{k} := 
\begin{cases}
-\mathrm{Im}[\psi(1-i\lambda)]/\lambda & k=0\\
-\mathrm{Re}[\psi(1-i\lambda)]-\log h & k=1\\
\zeta(2-k)-\lambda^2z_{k-2} & k>1
\end{cases},
\label{eq:zk_recurrence}
\end{equation}
where $\psi$ is the \emph{digamma function}, and where the parameter $\lambda$ is defined as
$$
\lambda:=\frac{d}{ch}.
$$
\end{theorem}

\begin{remark}
The first series of \eqref{eq:nearsing_EM} can be seen as a ``singular error'' that is a continuous extension from $d=0$ to $d>0$. The second series of \eqref{eq:nearsing_EM} can be seen as a ``jump error'' that corresponds to the jump discontinuity between $d>0$ and $d<0$; this ``jump error'' does not appear when $d=0$ as the integral $I\left[\frac{g( x)}{c^2 x^2}\right]$ is defined in the finite-part sense.
\end{remark}

In order to prove Theorem \ref{thm:nearsing_EM}, we will need the following two lemmas.

\begin{lemma}
\label{lem:digamma}
Define the family of rational zeta series
\begin{equation}
f_k(z) := \sum_{m=0}^\infty z^{2m}\zeta_h(2m+2-k),\quad k = 0,1,\dots
\label{eq:lemma_digammaseries}
\end{equation}
where $\zeta_h$ is the modified zeta function \eqref{eq:modified_zetafunc}. Then $f_k$ satisfies  the recurrence relations
\begin{equation}
\begin{aligned}
f_0(z) &= \tfrac{\psi(1+z)-\psi(1-z)}{2z},\\
f_1(z) &= -\tfrac{\psi(1+z)+\psi(1-z)}{2}-\log h,\\
f_{k}(z) &= \zeta(2-k)+z^2f_{k-2}(z), & k>1.
\end{aligned}
\label{eq:lemma_digamma_recurrence}
\end{equation}
where $\psi$ is the digamma function. Note that \eqref{eq:lemma_digamma_recurrence} implies that each $f_k(z)$, as defined in \eqref{eq:lemma_digammaseries}, can be analytically continued to a meromorphic function on $\mathbb{C}$.
\end{lemma}
\begin{proof}
The recurrence relation \eqref{eq:lemma_digamma_recurrence} is a direct consequence of the well-known expansion of the digamma function $\psi(1+z)$ as a rational zeta series \cite[eq.(6.3.14)]{abramowitz1968handbook}.
\end{proof}

\begin{lemma}
\label{lem:nearsing_EM_basis}
Using the notations (\ref{eq:notation_int}-\ref{eq:notation_trap}), we have, for any integer $k\geq0$,
\begin{equation}
R_h\Big[\frac{ x^{k}}{d^2+c^2 x^2}\Big] - I_R\Big[\frac{ x^{k}}{d^2+c^2 x^2}\Big] = 
\begin{cases}
\frac{1}{c^2}z_{k}h^{k-1}-(-1)^{k/2}\frac{d^{k-1}}{c^{k+1}}\frac{\pi}{2}, &k\text{ even} \\
\frac{1}{c^2}z_{k}h^{k-1}+(-1)^{\frac{k-1}{2}}\frac{d^{k-1}}{c^{k+1}}\log d, &k\text{ odd}
\end{cases}
\label{eq:lemma_right_nearsingEM}
\end{equation}
and
\begin{equation}
L_h\Big[\frac{ x^{k}}{d^2+c^2 x^2}\Big] - I_L\Big[\frac{ x^{k}}{d^2+c^2 x^2}\Big] = 
\begin{cases}
\frac{1}{c^2}z_{k}h^{k-1}-(-1)^{k/2}\frac{d^{k-1}}{c^{k+1}}\frac{\pi}{2}, &k\text{ even} \\
-\frac{1}{c^2}z_{k}h^{k-1}-(-1)^{\frac{k-1}{2}}\frac{d^{k-1}}{c^{k+1}}\log d, &k\text{ odd}
\end{cases},
\label{eq:lemma_left_nearsingEM}
\end{equation}
where $z_k$ is as defined in \eqref{eq:zk_recurrence}.
\end{lemma}
\begin{proof}
Note that the series
\begin{equation}
\frac{ x^{k}}{d^2+c^2 x^2} = \sum_{m=0}^\infty\frac{(-1)^md^{2m}}{c^{2m+2} x^{2m+2-k}}
\label{eq:lemma_binom_series}
\end{equation}
converges whenever $\left|\frac{d}{c x}\right|<1$. To prove \eqref{eq:lemma_right_nearsingEM}, we decompose the expression into a telescoping sum that consist of three components, $A_k, B_k$, and $C_k$, as follows
\begin{equation}
\begin{aligned}
R_h\Big[\frac{ x^{k}}{d^2+c^2 x^2}\Big] - &I_R\Big[\frac{ x^{k}}{d^2+c^2 x^2}\Big] =
\underbrace{R_h\Big[\frac{ x^{k}}{d^2+c^2 x^2}\Big] - \sum_{m=0}^\infty \frac{(-1)^md^{2m}}{c^{2m+2}}R_h\Big[\frac{1}{ x^{2m+2-k}}\Big]}_{:=A_k} \\
&\qquad + \underbrace{\sum_{m=0}^\infty \frac{(-1)^md^{2m}}{c^{2m+2}}\left(R_h\Big[\frac{1}{ x^{2m+2-k}}\Big] - f.p. I_R\Big[\frac{1}{ x^{2m+2-k}}\Big]\right)}_{:=B_k}\\
&\qquad + \underbrace{\sum_{m=0}^\infty \frac{(-1)^md^{2m}}{c^{2m+2}}\left( f.p. I_R\Big[\frac{1}{ x^{2m+2-k}}\Big]\right) - I_R\Big[\frac{ x^{k}}{d^2+c^2 x^2}\Big]}_{:=C_k}
\end{aligned}
\label{eq:lemma_decomp}
\end{equation}
First, the series \eqref{eq:lemma_binom_series} implies that
\begin{equation}
A_k = 0,\quad k\geq0
\label{eq:lemmaAk}
\end{equation} 
for $\lambda=\frac{d}{ch}<1$. Next, by applying Corollary \ref{thm:gen_EM_special} to each term of the series in $B_k$, we have
\begin{equation}
B_k = \sum_{m=0}^\infty\frac{(-1)^md^{2m}}{c^{2m+2}}\zeta_h(2m+2-k)h^{-2m-1+k} =\frac{h^{k-1}}{c^2}\sum_{m=0}^\infty\left(i\lambda\right)^{2m}\zeta_h(2m+2-k)
\label{eq:zk_series}
\end{equation}
which is a rational zeta series. Then substituting $z=i\lambda$ in Lemma \ref{lem:digamma} implies that
\begin{equation}
B_k = \frac{h^{k-1}}{c^2}z_{k}
\label{eq:lemmaBk}
\end{equation}
where $z_k$ is as defined in \eqref{eq:zk_recurrence}.

Finally, $C_k$ can be calculated by recursion as follows. Let
$$
C_k = \underbrace{\sum_{m=0}^\infty \frac{(-1)^md^{2m}}{c^{2m+2}}\left( f.p. I_R\Big[\frac{1}{ x^{2m+2-k}}\Big]\right)}_{:=C_k^{(1)}} - \underbrace{I_R\Big[\frac{ x^{k}}{d^2+c^2 x^2}\Big]}_{:=C_k^{(2)}}.
$$
Note that
$$
\begin{aligned}
C_{k+2}^{(2)} &= \int_0^a\frac{ x^{k+2}}{d^2+c^2 x^2}\d x\\
 &= \frac{1}{c^2}\left(\int_0^a x^{k}\d x -d^2\int_0^a\frac{ x^{k}}{d^2+c^2 x^2}\d x\right)\\
 &= \frac{1}{c^2}\left(\frac{a^{k+1}}{k+1}-d^2C_{k}^{(2)}\right)
\end{aligned}
$$
and similarly $C_{k+2}^{(1)} = \frac{1}{c^2}\left(\frac{a^{k+1}}{k+1}-d^2C_{k}^{(1)}\right)$. Thus $C_{k+2} = C_{k+2}^{(1)} - C_{k+2}^{(2)} = -\frac{d^2}{c^2}C_{k}$.
The initial terms of $C_k$ are $C_0 = -\frac{\pi}{2cd}$ and $C_1 = \frac{\log d}{c^2}$ by straightforward calculations. Therefore we have
\begin{equation}
C_k = 
\begin{cases}
-(-1)^{k/2}\frac{d^{k-1}}{c^{k+1}}\frac{\pi}{2}, &k\text{ even} \\
(-1)^{\frac{k-1}{2}}\frac{d^{k-1}}{c^{k+1}}\log d, &k\text{ odd}
\end{cases}.
\label{eq:lemmaCk}
\end{equation}
Substituting (\ref{eq:lemmaAk},\ref{eq:lemmaBk},\ref{eq:lemmaCk}) into \eqref{eq:lemma_decomp} yields \eqref{eq:lemma_right_nearsingEM}. Note that for fixed $h$, both sides of \eqref{eq:lemma_right_nearsingEM} are analytic functions of $\lambda=\frac{d}{ch}$, thus the identity also holds for $\lambda\geq1$ by analytic continuation. 

The proof of \eqref{eq:lemma_left_nearsingEM} is similar to that of \eqref{eq:lemma_right_nearsingEM}.
\end{proof}

We are now ready to prove Theorem \ref{thm:nearsing_EM}.

\begin{proof}[Proof of Theorem \ref{thm:nearsing_EM}] Using the Maclaurin series of $g( x)$, we have
\begin{equation}
E_h\Big[\frac{g( x)}{d^2+c^2 x^2}\Big] = \sum_{k=0}^\infty\frac{g^{(k)}(0)}{k!}E_h\Big[\frac{ x^k}{d^2+c^2 x^2}\Big].
\end{equation}
Each term of the above series can then be substituted by \eqref{eq:lemma_right_nearsingEM} and \eqref{eq:lemma_left_nearsingEM}, which yields \eqref{eq:nearsing_EM}.
\end{proof}

\begin{remark}
When the function $g( x)$ can be analytically continued to a neighborhood of $ x=0$ in $\CC$, the generalized Euler-Maclaurin formula \eqref{eq:nearsing_EM} has a closed-form expression given by
\begin{align}
E_h\Big[\frac{g( x)}{d^2+c^2 x^2}\Big] = \frac{1}{c^2}\mathrm{Re}\Big[\frac{g(i\lambda h)-g(0)}{(i\lambda h)^2}\Big]h + \left(\frac{\pi}{cd}-\frac{2z_0}{c^2h}\right)\mathrm{Re}[g(i\lambda h)]
\label{eq:spectral_digamma_rule_centered}
\end{align}
For a derivation of \eqref{eq:spectral_digamma_rule_centered}, first of all, note that by the recurrence relation \eqref{eq:zk_recurrence} the coefficients $z_{2k}$ in \eqref{eq:nearsing_EM} satisfy $z_{2k} = (-\lambda^2)^k(z_0+\frac{1}{2\lambda^2})$ when $k\geq1$, thus
        \begin{equation}
        \begin{aligned}
        \sum_{k=0}^\infty\frac{g^{(2k)}(0)}{(2k)!}h^{2k-1}\,z_{2k} &= \frac{g(0)z_0}{h} + \sum_{k=1}^\infty\frac{g^{(2k)}(0)}{(2k)!}h^{2k-1}\,(-\lambda^2)^k(z_0+\tfrac{1}{2\lambda^2})\\
        &= -\frac{g(0)}{2\lambda^2h} + \left(\frac{z_0}{h}+\frac{1}{2\lambda^2h}\right)\sum_{k=0}^\infty\frac{g^{(2k)}(0)}{(2k)!}(i\lambda h)^{2k}\\
        &= -\frac{g(0)}{2\lambda^2h} + \left(\frac{z_0}{h}+\frac{1}{2\lambda^2h}\right)\mathrm{Re}[g(i\lambda h)]\\
        &= \frac12\mathrm{Re}\Big[\frac{g(i\lambda h)-g(0)}{(i\lambda h)^2}\Big]h + \frac{z_0}{h}\mathrm{Re}[g(i\lambda h)].
        \end{aligned}
        \label{eq:rmk2_derive1}
        \end{equation}
Second of all, the second series in \eqref{eq:nearsing_EM} can be expressed as
        \begin{equation}
        \sum_{k=0}^\infty (-\tfrac{d^2}{c^2})^{k} \frac{g^{(2k)}(0)}{(2k)!} = \mathrm{Re}[g(id/c)] = \mathrm{Re}[g(i\lambda h)].
        \label{eq:rmk2_derive2}
        \end{equation}
Combining (\ref{eq:rmk2_derive1}--\ref{eq:rmk2_derive2}) yields \eqref{eq:spectral_digamma_rule_centered}.
To see the connection of \eqref{eq:spectral_digamma_rule_centered} with the existing generalized Euler-Maclaurin formulas in the literature, consider its limit as $d\to0$ (thus $\lambda h = d/c\to0$)
\begin{equation}
\lim_{d\to0} E_h\Big[\frac{g( x)}{d^2+c^2 x^2}\Big] = \frac{1}{c^2}\left(\frac{g''(0)}{2}h -\frac{2\zeta(2)}{h}g(0)\right) +  \lim_{d\to0}\frac{\pi}{cd}\,[g(0)+O(d^2)]
\end{equation}
This agrees with the finite-part quadrature correction of \cite[eq.(3.2)-(3.4)]{sidi2013compact}~:
\begin{equation}
E_h\Big[\frac{g( x)}{ x^2}\Big] = \frac{g''(0)}{2}h -\frac{2\zeta(2)}{h}g(0) = \left(-\frac{2\zeta(2)}{h^2}\frac{g(0)}{0!} - \frac{2\zeta(0)}{h^0}\frac{g''(0)}{2!}\right)h
\end{equation}
and the fact that there is a jump $\frac{\pi}{d}\,g(0)$ between the classical and principal value integrals.
\end{remark}

\subsection{Euler-Maclaurin formula for an off-mesh near singularity}

The generalized Euler-Maclaurin formula of Theorem \ref{thm:nearsing_EM} can be further extended to an integral $I\big[\frac{g( x)}{d^2+c^2( x- x_s)^2}\big]$ where $x_s\neq0$, whose near-singularity at $x_s$ is no long a mesh point of the Trapezoidal rule $T_h$; however, without loss of generality, we assume $x_s$ is small enough such that $x=0$ is still the closest mesh point from the near-singularity.

\begin{theorem}[Off-mesh Euler-Maclaurin for near-singular integral] 
\label{thm:nearsing_EM_offmesh}
Given $g(x)\in C^{\infty}([a,a])$ and $c,d\neq0$,
\begin{align}
E_h\Big[\frac{g( x)}{d^2+c^2( x- x_s)^2}\Big] \sim -\frac{1}{c^2}\sum_{k=0}^\infty\frac{g^{(k)}( x_s)}{k!}h^{k-1} p_{k,s} + \frac{\pi}{cd}\sum_{k=0}^\infty \frac{g^{(2k)}( x_s)}{(2k)!}\left(-\frac{d^2}{c^2}\right)^{k}
\label{eq:nearsing_EM_offmesh}
\end{align}
as $h\to0$ and $d\to0$, where the numbers $p_{k,s}$ satisfy the recurrence relation
\begin{equation}
p_{k,s} = 
\begin{cases}
-\mathrm{Im}[\psi(1-s-i\lambda)+\psi(1+s-i\lambda)]/\lambda & k=0\\
-\mathrm{Re}[\psi(1-s-i\lambda)-\psi(1+s-i\lambda)] & k=1\\
-(-s)^{k-2} - \lambda^2p_{k-2,s} & k\geq2
\end{cases}
\label{eq:pks_recurrence}
\end{equation}
where $\psi$ is the \emph{digamma function}, and where the parameter $\lambda=d/(ch)$.
\end{theorem}

Note that $\lim_{s\to0}p_{k,s} = z_k+(-1)^kz_k$, so the formula \eqref{eq:nearsing_EM_offmesh} reduces to \eqref{eq:nearsing_EM} as $ x_s\to0$. In order to prove Theorem \ref{thm:nearsing_EM_offmesh}, we first need the next three lemmas.

\begin{lemma}
\label{lem:nearsing_EM_basis_shfted}
Using the notations \eqref{eq:notation_int} and \eqref{eq:notation_trap_shift}, we have, for any integer $k\geq0$,
\begin{equation}
R_{h,s}\Big[\frac{ x^{k}}{d^2+c^2 x^2}\Big] - I_R\Big[\frac{ x^{k}}{d^2+c^2 x^2}\Big] = 
\begin{cases}
\frac{1}{c^2}z_{k,s}h^{k-1}-(-1)^{k/2}\frac{d^{k-1}}{c^{k+1}}\frac{\pi}{2}, &k\text{ even} \\
\frac{1}{c^2}z_{k,s}h^{k-1}+(-1)^{\frac{k-1}{2}}\frac{d^{k-1}}{c^{k+1}}\log d, &k\text{ odd}
\end{cases}
\label{eq:lemma_right_nearsingEM_shifted}
\end{equation}
and
\begin{equation}
L_{h,s}\Big[\frac{ x^{k}}{d^2+c^2 x^2}\Big] - I_L\Big[\frac{ x^{k}}{d^2+c^2 x^2}\Big] = 
\begin{cases}
\frac{1}{c^2}z_{k,-s}h^{k-1}-(-1)^{k/2}\frac{d^{k-1}}{c^{k+1}}\frac{\pi}{2},&k\text{\,even}\\
-\frac{1}{c^2}z_{k,-s}h^{k-1}-(-1)^{\frac{k-1}{2}}\frac{d^{k-1}}{c^{k+1}}\log d,&k\text{\,odd}
\end{cases}
\label{eq:lemma_left_nearsingEM_shifted}
\end{equation}
where the coefficients $z_{k,s}$ satisfy the recurrence relation 
\begin{equation}
z_{k,s} = 
\begin{cases}
-\mathrm{Im}\Big[\psi(1+s-i\lambda)\Big]/\lambda & k=0\\
-\mathrm{Re}[\psi(1+s-i\lambda)]-\log h & k=1\\
\zeta(2-k,1+s) - \lambda^2\,z_{k-2,s} & k>1
\end{cases}
\label{zks_recur}
\end{equation}
\end{lemma}
\begin{proof}
The proof of Lemma \ref{lem:nearsing_EM_basis_shfted} follows similar arguments as in the proof of Lemma \ref{lem:nearsing_EM_basis}. In particular, based on Theorem \ref{thm:gen_EM} a series expression for $z_{k,s}$ similar to \eqref{eq:zk_series} can be obtained, giving
\begin{equation}
z_{k,s} = \sum_{m=0}^\infty\left(i\lambda\right)^{2m}\zeta_h(2m+2-k,1+s).
\label{eq:zks_series}
\end{equation}
Then one can show that $z_{k,s}$ defined by the series \eqref{eq:zks_series} satisfies the recurrence relation \eqref{zks_recur} by substituting $z=i\lambda$ in Lemma \ref{lem:digamma_shifted} below. 
\end{proof}

\begin{lemma}
\label{lem:digamma_shifted}
Define the family of rational zeta series
\begin{equation}
f_{k,s}(z) := \sum_{m=0}^\infty z^{2m}\zeta_h(2m+2-k,1+s),\quad k = 0,1,\dots
\label{eq:lemma_digammaseries_shifted}
\end{equation}
where $\zeta_h(z,s)$ is the modified Hurwitz zeta function \eqref{eq:modified_hurwitzzeta}. Then $f_k$ satisfies  the recurrence relation
\begin{equation}
\begin{aligned}
f_{0,s}(z) &= \tfrac{\psi(1+s+z)-\psi(1+s-z)}{2z},\\
f_{1,s}(z) &= -\tfrac{\psi(1+s+z)+\psi(1+s-z)}{2}-\log h,\\
f_{k,s}(z) &= \zeta(2-k,1+s)+z^2f_{k-2,s}(z) & \text{for }k>1.
\end{aligned}
\label{eq:lemma_digamma_shifted_recurrence}
\end{equation}
where $\psi$ is the digamma function. Note that \eqref{eq:lemma_digamma_shifted_recurrence} implies that each $f_{k,s}(z)$, as defined in  \eqref{eq:lemma_digammaseries_shifted}, can be analytically continued to a meromorphic function on $\mathbb{C}$.
\end{lemma}
\begin{proof}
The connection between the rational zeta series \eqref{eq:lemma_digammaseries_shifted} and the digamma function $\psi$ can
be established by the relation \cite[eq.(25.11.12)]{olver2010nist}
\begin{equation}
\zeta(n+1,a)=\frac{(-1)^{n+1}\psi^{(n)}(a)}{n!}, \quad n=1,2,3,\dots
\label{eq:hurtwitz_digamma_relation}
\end{equation}
\end{proof}

\begin{lemma}
\label{lem:hurwitz_identity}
For any integer $k\geq0$,
\begin{equation}
\zeta(-k,1+s)+(-1)^k\zeta(-k,1-s) = -s^k
\end{equation}
\end{lemma}
\begin{proof}
Let $B_n(x)$ be the Bernoulli polynomial of degree $n$. By the identities \cite{abramowitz1968handbook}
$$
\begin{aligned}
\zeta(-n,x) &= -\frac{B_{n+1}(x)}{n+1},\\
B_n(1+x) &= B_n(x) + nx^{n-1}\\
B_n(1-x) &= (-1)^nB_n(x),
\end{aligned}
$$
we have
$$
\begin{aligned}
\zeta(-k,1+s)+(-1)^k\zeta(-k,1-s) &= -\frac{1}{k+1}\Big(B_{k+1}(1+s) + (-1)^kB_{k+1}(1-s)\Big)\\
&= -\frac{1}{k+1}\Big(\big(B_{k+1}(s) + (k+1)s^k\big) - B_{k+1}(s)\Big)\\
&= -s^k
\end{aligned}
$$
\end{proof}

We are now ready to prove Theorem \ref{thm:nearsing_EM_offmesh}.

\begin{proof}[Proof of Theorem \ref{thm:nearsing_EM_offmesh}]
If we define $\hat{ x}: =  x- x_s$, then
\begin{equation}
E_h\Big[\frac{g( x)}{d^2+c^2( x- x_s)^2}\Big] = E_{h,-s}\Big[\frac{g(\hat x+ x_s)}{d^2+c^2\hat{ x}^2}\Big]= \sum_{k=0}^\infty\frac{g^{(k)}( x_s)}{k!}E_{h,-s}\Big[\frac{\hat x^k}{d^2+c^2\hat x^2}\Big]
\end{equation}
then substituting (\ref{eq:lemma_right_nearsingEM_shifted}--\ref{eq:lemma_left_nearsingEM_shifted}) into the above series yields
\begin{equation}
\begin{aligned}
E_h\Big[\frac{g( x)}{d^2+c^2( x- x_s)^2}\Big] = -\frac{1}{c^2}\sum_{k=0}^\infty\frac{g^{(k)}( x_s)}{k!}h^{k-1}(z_{k,-s}+(-1)^kz_{k,s})\\
 + \frac{\pi}{cd}\sum_{k=0}^\infty \frac{g^{(2k)}( x_s)}{(2k)!}\Big(-\frac{d^2}{c^2}\Big)^{k}.
\end{aligned}
\end{equation}
Then using the recurrence \eqref{zks_recur} and Lemma \ref{lem:hurwitz_identity}, it immediately follows that the coefficients
\begin{equation}
z_{k,-s}+(-1)^kz_{k,s} \equiv p_{k,s}
\end{equation}
satisfy the recurrence relation \eqref{eq:pks_recurrence}.
\end{proof}

\begin{remark}
Similar to \eqref{eq:spectral_digamma_rule_centered}, when $g(x)$ is analytic in a neighborhood of $ x= x_s$ in $\CC$ the generalized Euler-Maclaurin formula \eqref{eq:nearsing_EM_offmesh} has the closed form expression
\begin{equation}
\begin{aligned}
E_h\Big[\frac{g( x)}{d^2+c^2( x- x_s)^2}\Big] &=-\frac{1}{c^2h}\Bigg\{\left(p_{0,s}+\frac{1}{s^2+\lambda^2}\right)\mathrm{Re}[g( x_s+i\lambda h)] -\frac{g(x_s-sh)}{\lambda^2+ s^2}\\
&\quad +\left(p_{1,s}-\frac{s}{s^2+\lambda^2}\right)\frac{\mathrm{Im}[g( x_s+i\lambda h)]}{\lambda}\Bigg\}+\frac{\pi}{cd}\mathrm{Re}[g( x_s+i\lambda h)].
\end{aligned}
\label{eq:spectral_digamma_rule}
\end{equation}
This expression can be derived as follows.
\begin{enumerate}
\item The recurrence relation \eqref{eq:pks_recurrence} implies that $p_{k,s}$ has the following closed-form expression
\begin{equation}
\begin{aligned}
p_{2k,s} &= -\frac{s^{2k}-(-\lambda^2)^k}{s^2+\lambda^2}+(-\lambda^2)^kp_{0,s},\\
p_{2k+1,s} &= \frac{s^{2k+1}-(-\lambda^2)^ks}{s^2+\lambda^2}+(-\lambda^2)^kp_{1,s}.
\end{aligned}
\label{eq:pks_closedform}
\end{equation}
Substituting \eqref{eq:pks_closedform} into the first series on the right-hand side of \eqref{eq:nearsing_EM_offmesh}, we have
\begin{equation}
\begin{aligned}
\sum_{k=0}^\infty&\frac{g^{(k)}( x_s)}{k!}h^{k-1} p_{k,s}\\
&=\sum_{k=0}^\infty\frac{g^{(2k)}( x_s)}{(2k)!}h^{2k-1}\left(\frac{(-\lambda^2)^k-s^{2k}}{s^2+\lambda^2}+(-\lambda^2)^kp_{0,s}\right)\\
&\quad+\sum_{k=0}^\infty\frac{g^{(2k+1)}( x_s)}{(2k+1)!}h^{2k}\left(\frac{s^{2k+1}-(-\lambda^2)^ks}{s^2+\lambda^2}+(-\lambda^2)^kp_{1,s}\right)\\
&=\frac{1}{(s^2+\lambda^2)h}\left(\mathrm{Re}[g( x_s+i\lambda h)]-g( x_s-s\,h) - s\frac{\mathrm{Im}[g( x_s+i\lambda h)]}{\lambda}\right)\\
&\quad + \frac{p_{0,s}}{h}\mathrm{Re}[g( x_s+i\lambda h)]+\frac{p_{1,s}}{h}\frac{\mathrm{Im}[g( x_s+i\lambda h)]}{\lambda}.
\end{aligned}
\end{equation}
\item The second series on the right-hand side of \eqref{eq:nearsing_EM_offmesh} can be expressed as
\begin{equation}
\sum_{k=0}^\infty (-\tfrac{d^2}{c^2})^{k} \frac{g^{(2k)}( x_s)}{(2k)!} = \mathrm{Re}[g( x_s+id/c)] = \mathrm{Re}[g( x_s+i\lambda h)].
\end{equation}
Now, combining the above expressions, \eqref{eq:nearsing_EM_offmesh} becomes \eqref{eq:spectral_digamma_rule}.
\end{enumerate}
\end{remark}

\subsection{New singular Trapezoidal rule for off-mesh singularity}

As a consequence of Theorem \ref{thm:nearsing_EM_offmesh}, we obtain in the following corollary an off-mesh version of the generalized Euler-Maclaurin formula for a finite-part (hypersingular) integral, which seems to be new.

\begin{corollary}[Off-mesh hypersingular Euler-Maclaurin formula]
\label{thm:sing_EM_offmesh}
\begin{equation}
\begin{aligned}
E_h\Big[\frac{g( x)}{( x- x_s)^2}\Big] &= \frac{g(0)-g( x_s)-g'( x_s)(- x_s)}{ x_s^2}h\\
&\quad -\frac{\zeta(2,1-s)+\zeta(2,1+s)}{h}g( x_s)\\
&\quad -(-\psi(1-s)+\psi(1+s))g'( x_s)
\end{aligned}
\label{eq:sing_EM_offmesh}
\end{equation}
\end{corollary}
\begin{proof}
Using the definition \eqref{eq:pks_recurrence} and the relation \eqref{eq:hurtwitz_digamma_relation}, we have
\begin{equation}
\begin{aligned}
&\lim_{\lambda\to0}p_{0,s}=\zeta(2,1-s)+\zeta(2,1+s)\\
&\lim_{\lambda\to0}p_{1,s}= -\psi(1-s)+\psi(1+s)\\
&\lim_{\lambda\to0}\mathrm{Re}\Big[\tfrac{g( x_s+i\lambda h)}{i\lambda h}\Big] = g'( x_s)
\end{aligned}
\end{equation}
Substituting the above limits into \eqref{eq:spectral_digamma_rule} as $d\to0$ (hence $\lambda\to0$) yields \eqref{eq:sing_EM_offmesh}, where the jump $\lim_{d\to0}\frac{\pi}{cd}\mathrm{Re}[g( x_s+i\lambda h)]$ is ignored to be consistent with the value of the finite-part integral.
\end{proof}

\begin{remark}
The following quadrature for a finite-part integral is presented in \cite[eq.(3.10)]{sidi2013compact}
\begin{equation}
I\Big[\frac{g( x)}{ x^2}\Big] \approx \tilde{T}_{h,1/2}\Big[\frac{g( x)}{ x^2}\Big] - \pi^2h^{-1}g(h/2),
\label{eq:hyperEM_sidi}
\end{equation}
which is spectrally accurate; here, $\tilde{T}$ represents an ordinary Trapezoidal rule (not punctured). One can derive this formula as a consequence of Corollary \ref{thm:sing_EM_offmesh}.
Let $s\to 1/2$ (hence $ x_s\to h/2$) in \eqref{eq:sing_EM_offmesh} and using the fact that \cite[eq.(25.11.8), eq.(5.15.5)]{olver2010nist}
\begin{equation}
\begin{aligned}
&\lim_{s\to\frac12}\zeta(2,1-s)+\zeta(2,1+s) = \pi^2-4\\
&\lim_{s\to\frac12}-\psi(1-s)+\psi(1+s)= 2
\end{aligned}
\end{equation}
we arrive at the expression
\begin{equation}
E_h\Big[\frac{g( x)}{( x-h/2)^2}\Big] = \frac{g(0)}{(h/2)^2}h - \pi^2h^{-1}g(h/2)
\label{eq:tmp_eq1_sidi_proof}
\end{equation}
If the $g(0)$ term in \eqref{eq:tmp_eq1_sidi_proof} is included back into the punctured Trapezoidal rule on the left-hand side to yield the ordinary Trapezoidal rule, then the error expression becomes
\begin{equation}
\tilde{E}_h\Big[\frac{g( x)}{( x-h/2)^2}\Big] \equiv \tilde{E}_{h,1/2}\Big[\frac{g( x)}{ x^2}\Big] = - \pi^2h^{-1}g(h/2),
\end{equation}
where $\tilde{E}$ denotes the error for $\tilde{T}$ in \eqref{eq:hyperEM_sidi}, which is the consistency we wanted to show.
\end{remark}

\section{Numerical Examples}
\label{sc:numerical}

We present convergence results for the generalized Euler-Maclaurin formulas of Section \ref{sc:err_flat}.

Let $\mathrm{Ei}(z)$ be the exponential integral and consider the near-singular integral 
\begin{equation}
I=\int_{-1}^1\frac{de^x}{d^2+x^2}\,\d x
\label{eq:test1_integral}
\end{equation}
whose exact value $$I=\mathrm{Im}\left\{e^{id}\big(\mathrm{Ei}(1-id)-\mathrm{Ei}(-1-id)\big)\right\}$$ is well-defined for all $d\in\RR$. Note that the integral has a well-defined limit $I\to\pm\pi$ as $d\to0^{\pm}$. Figure \ref{fig:conv_nearsing_EM} presents convergence results of the corrected trapezoidal rule based on \eqref{eq:nearsing_EM}, with $h=2/N$, applied to the integral \eqref{eq:test1_integral} for different values of $d$, where the near-singular point $x=0$ is a quadrature node; the correction is based on the closed-form expression \eqref{eq:spectral_digamma_rule_centered}, so the corrected quadrature converges exponentially. Machine precision is reached around $N=100$ independent of the size of $d$. On the contrary, the convergence of the regular Trapezoidal rule without correction slows down significantly as $d\to0$ and the integral becomes closer to singular.

\begin{figure}[h]
\includegraphics[width=0.95\textwidth]{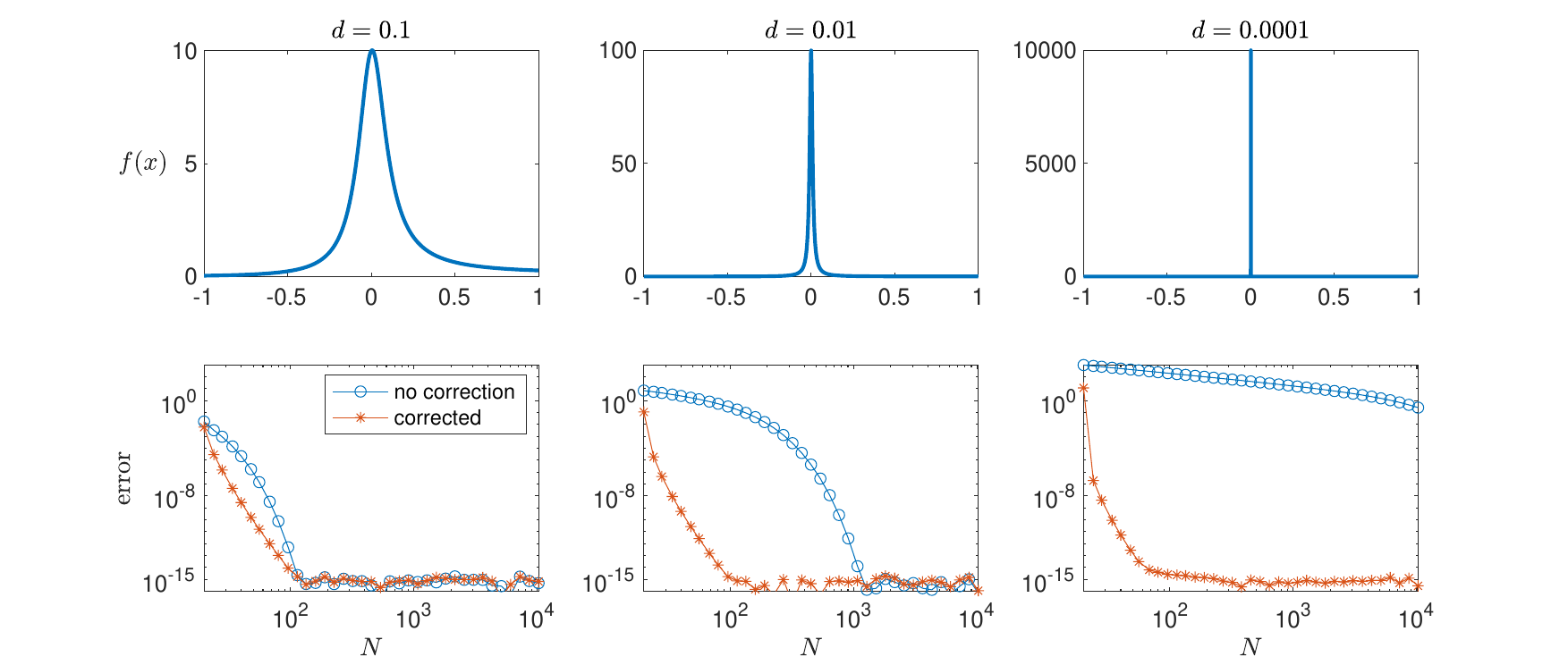}
\caption{Convergence of the Trapezoidal rule (mesh size $h=2/N$), with and without corrections, applied to the integral of $f(x)=\frac{de^x}{d^2+x^2}$, for $d=0.1, 0.01, 0.0001$. The near-singular correction is based on the formula \eqref{eq:spectral_digamma_rule_centered}.}
\label{fig:conv_nearsing_EM}
\end{figure}

Similar convergence results are observed for the off-mesh near-singularity in Figure \ref{fig:conv_nearsing_EM_offmesh}, where the corrected trapezoidal rule based on \eqref{eq:nearsing_EM_offmesh} and the closed-form correction formula \eqref{eq:spectral_digamma_rule} is applied to the integral
$$I=\int_{-1}^1\frac{de^x}{d^2+c^2(x-x_s)^2}\,\d x=\mathrm{Im}\left\{e^{x_s+\frac{id}{c}}\big(\mathrm{Ei}(1-x_s-\tfrac{id}{c})-\mathrm{Ei}(-1-x_s-\tfrac{id}{c})\big)\right\}$$
with $x_s = 0.1$. Note that $x_s$ is not a quadrature node in this example.

\begin{figure}[h]
\includegraphics[width=0.95\textwidth]{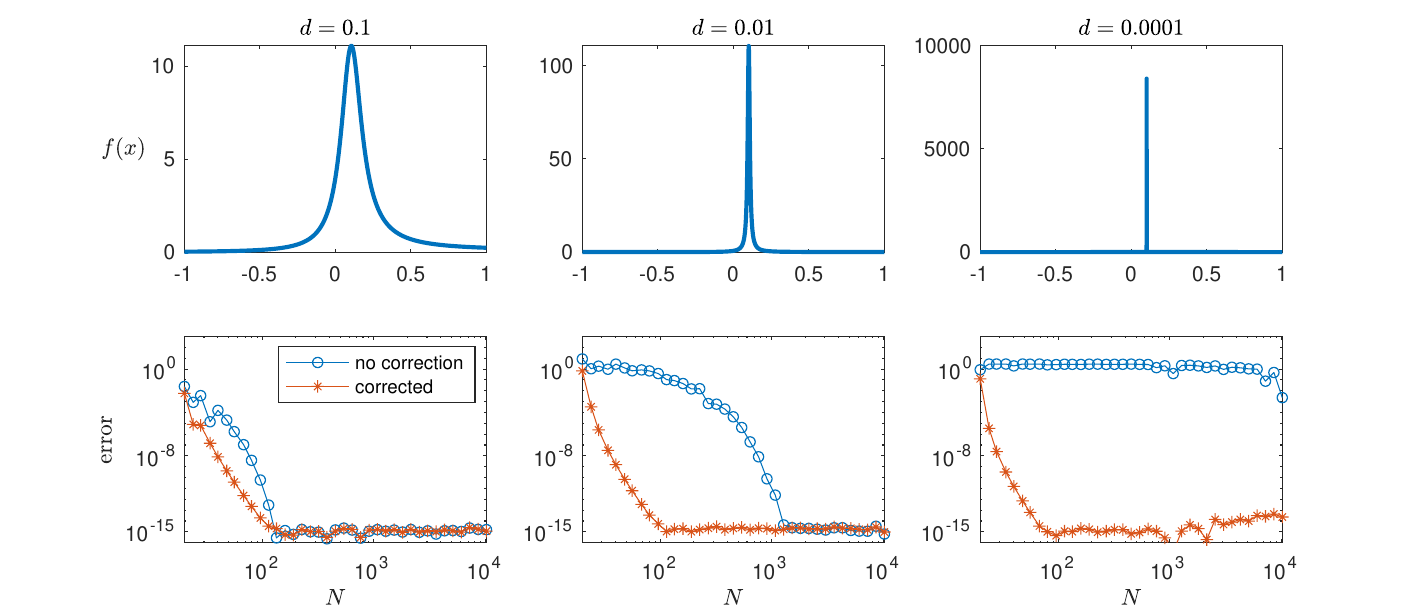}
\caption{Similar to Figure \ref{fig:conv_nearsing_EM} but for the integral of $f(x)=\frac{de^x}{d^2+c^2(x-x_s)^2}$ with near-singularity at $x_s = 0.1$ (not a quadrature node). The correction is based on the formula \eqref{eq:spectral_digamma_rule}.}
\label{fig:conv_nearsing_EM_offmesh}
\end{figure}

\begin{remark}
For a general integral where the analytic expression of $g(x)$ is unknown or where $g(x)$ cannot be analytically continued to a neighborhood of the near-singular point, a truncation of the error corrections \eqref{eq:nearsing_EM} and \eqref{eq:nearsing_EM_offmesh} can be computed by approximating $g^{(k)}(0)$ or $g^{(k)}(x_s)$ using finite difference formulas on a small stencil centered at $x=0$. In practice, we find that convergence results that are almost identical to Figures \ref{fig:conv_nearsing_EM} and \ref{fig:conv_nearsing_EM_offmesh} can be achieved by including correction terms in \eqref{eq:nearsing_EM} and \eqref{eq:nearsing_EM_offmesh} that involve up to the $6^{\rm th}$ derivative of $g(x)$. Matlab code for these computations are available at {\tt https://github.com/bobbielf2/PolygammaTrap}.
\label{rmk:code}
\end{remark}

\section{Conclusion}
\label{sc:conclude}

We have presented generalized Euler-Maclaurin formulas for near-singular integrals of the form $$\int \frac{f(x)}{d^2+c^2(x-x_s)^2}\,\d x,$$ which achieve high-order convergence for any values of $d$. 

A future direction is to further extend the current work to integrals of the form $$\int \frac{f(x)}{(d^2+c^2(x-x_s)^2)^p}\,\d x, ~~~p>1$$ and develop near-singular quadrature rules for integral operators such as \eqref{eq:nearsing_op}, which are near-singular integrals on curves in $\RR^2$ that arise in the solution of partial differential equations. Another future direction is to derive analogous generalizations of the Euler-Maclaurin formula in $\RR^2$ and apply them to surface integrals in $\RR^3$, which will be an extension of the singular quadrature method of \cite{wu2021corrected,wu2023unified}.

\bibliographystyle{plain}
\bibliography{bobi_quadr}

\begin{thebibliography}{10}

\bibitem{abramowitz1968handbook}
Milton Abramowitz and Irene~A Stegun.
\newblock {\em Handbook of mathematical functions with formulas, graphs, and
  mathematical tables}, volume~55.
\newblock US Government printing office, 1968.

\bibitem{alpert1999hybrid}
Bradley~K Alpert.
\newblock Hybrid {G}auss-trapezoidal quadrature rules.
\newblock {\em SIAM Journal on Scientific Computing}, 20(5):1551--1584, 1999.

\bibitem{lsc2d}
Alex Barnett, Bowei Wu, and Shravan Veerapaneni.
\newblock Spectrally accurate quadratures for evaluation of layer potentials
  close to the boundary for the 2{D} {S}tokes and {L}aplace equations.
\newblock {\em SIAM Journal on Scientific Computing}, 37(4):B519--B542, January
  2015.

\bibitem{davis1984methods}
P.J. Davis and P.~Rabinowitz.
\newblock {\em Methods of Numerical Integration}.
\newblock Computer Science and Applied Mathematics. A series of monographs and
  textbooks. Academic Press, 1984.

\bibitem{euler1741inventio}
Leonhard Euler.
\newblock Inventio summae cuiusque seriei ex dato termino generali.
\newblock {\em Commentarii academiae scientiarum Petropolitanae}, 8:9--22,
  1741.
\newblock English translation as Finding the sum of any series from a given
  general term, Euler Archive. \url{http://eulerarchive.maa.org}, no. 47.

\bibitem{javed2016euler}
Mohsin Javed and Lloyd~N Trefethen.
\newblock {E}uler--{M}aclaurin and {G}regory interpolants.
\newblock {\em Numerische Mathematik}, 132(1):201--216, 2016.

\bibitem{kapur1997high}
Sharad Kapur and Vladimir Rokhlin.
\newblock High-order corrected trapezoidal quadrature rules for singular
  functions.
\newblock {\em SIAM Journal on Numerical Analysis}, 34(4):1331--1356, 1997.

\bibitem{kress2014linear}
Rainer Kress.
\newblock {\em Linear Integral Equations}, volume~82 of {\em Applied
  Mathematical Sciences}.
\newblock Springer-Verlag New York, 3 edition, 2014.

\bibitem{lyness1994finite}
JN~Lyness.
\newblock Finite-part integrals and the {E}uler-{M}aclaurin expansion.
\newblock In {\em Approximation and Computation: A Festschrift in Honor of
  Walter Gautschi: Proceedings of the Purdue Conference, December 2--5, 1993},
  pages 397--407. Springer, 1994.

\bibitem{maclaurin1742treatise}
Colin Maclaurin.
\newblock {\em A Treatise of Fluxions}.
\newblock T.W.andT. Ruddimans, Edinburgh, 1742.

\bibitem{navot1961extension}
Israel Navot.
\newblock An extension of the {E}uler-{M}aclaurin summation formula to
  functions with a branch singularity.
\newblock {\em Journal of Mathematics and Physics}, 40(1-4):271--276, 1961.

\bibitem{navot1962further}
Israel Navot.
\newblock A further extension of the {E}uler-{M}aclaurin summation formula.
\newblock {\em Journal of Mathematics and Physics}, 41(1-4):155--163, 1962.

\bibitem{ninham1966generalised}
BW~Ninham.
\newblock Generalised functions and divergent integrals.
\newblock {\em Numerische Mathematik}, 8:444--457, 1966.

\bibitem{nitsche2021tcfd}
Monika Nitsche.
\newblock Evaluation of near-singular integrals with application to vortex
  sheet flow.
\newblock {\em Theoretical and Computational Fluid Dynamics}, 35(5):581--608,
  July 2021.

\bibitem{nitsche2025corrected}
Monika Nitsche, Bowei Wu, and Ling Xu.
\newblock Corrected trapezoidal rules for near-singular surface integrals
  applied to 3d interfacial {S}tokes flow.
\newblock {\em arXiv preprint arXiv:2504.01144}, 2025.

\bibitem{olver2010nist}
F.W.J. Olver and {National Institute of Standards and Technology (U.S.)}.
\newblock {\em NIST Handbook of Mathematical Functions Hardback and CD-ROM}.
\newblock Cambridge University Press, 2010.

\bibitem{sidi2004euler}
Avram Sidi.
\newblock {E}uler--{M}aclaurin expansions for integrals with endpoint
  singularities: a new perspective.
\newblock {\em Numerische Mathematik}, 98:371--387, 2004.

\bibitem{sidi2012algebraic}
Avram Sidi.
\newblock {E}uler--{M}aclaurin expansions for integrals with arbitrary
  algebraic endpoint singularities.
\newblock {\em Mathematics of Computation}, 81(280):2159--2173, 2012.

\bibitem{sidi2013compact}
Avram Sidi.
\newblock Compact numerical quadrature formulas for hypersingular integrals and
  integral equations.
\newblock {\em Journal of Scientific Computing}, 54(1):145--176, 2013.

\bibitem{sidi1988quadrature}
Avram Sidi and Moshe Israeli.
\newblock Quadrature methods for periodic singular and weakly singular
  {F}redholm integral equations.
\newblock {\em Journal of Scientific Computing}, 3(2):201--231, 1988.

\bibitem{wu2021corrected}
Bowei Wu and Per-Gunnar Martinsson.
\newblock Corrected trapezoidal rules for boundary integral equations in three
  dimensions.
\newblock {\em Numerische Mathematik}, 149:1025--1071, 2021.

\bibitem{wu2021zeta}
Bowei Wu and Per-Gunnar Martinsson.
\newblock Zeta correction: a new approach to constructing corrected trapezoidal
  quadrature rules for singular integral operators.
\newblock {\em Advances in Computational Mathematics}, 47(3):1--21, 2021.

\bibitem{wu2023unified}
Bowei Wu and Per-Gunnar Martinsson.
\newblock A unified trapezoidal quadrature method for singular and
  hypersingular boundary integral operators on curved surfaces.
\newblock {\em SIAM Journal on Numerical Analysis}, 61(5):2182--2208, 2023.

\end{thebibliography}

\end{document}